\documentclass{article}

\usepackage[english]{babel}
\usepackage[utf8]{inputenc}

\usepackage{amsmath}
\usepackage{amsfonts}
\usepackage{graphicx}
\usepackage[colorinlistoftodos]{todonotes}
\usepackage{blindtext}
\usepackage{amssymb}
\usepackage{authblk}
\usepackage{amsthm}
\numberwithin{equation}{section}
\newtheorem{theorem}{Theorem}[section]

\newtheorem{lemma}[theorem]{Lemma}
\newtheorem{proposition}{Proposition}[section]
\newtheorem{example}{Example}[section]
\theoremstyle{remark}
\newtheorem*{remark}{Remark}

\theoremstyle{definition}

\newcommand{\mytodo}[2][]{{%
 \let\marginpar\marginnote
 \reversemarginpar
 \renewcommand{\baselinestretch}{0.8}%
 \todo[#1]{#2}}}

\title{Hahn polynomials and the Burnside process}

\author{Persi Diaconis\thanks{Department of Statistics and Mathematics, Stanford University} \thanks{Research supported by NSF grant DMS 1954042} \and Chenyang Zhong\thanks{Department of Statistics, Stanford University} $^\dagger$}

\date{}

\begin{document}
\maketitle

\begin{abstract}
We study a natural Markov chain on $\{0,1,\cdots,n\}$ with eigenvectors the Hahn polynomials. This explicit diagonalization makes it possible to get sharp rates of convergence to stationarity. The process, the Burnside process, is a special case of the celebrated `Swendsen-Wang' or `data augmentation' algorithm. The description involves the beta-binomial distribution and Mallows model on permutations. It introduces a useful generalization of the Burnside process.

\emph{Dedicated to the memory of Richard Askey}
\end{abstract}

\section{Introduction}\label{Sect.1}
Over the past 100 years, physicists, geneticists, statisticians and probabilists have found positive symmetric operators (Markov chains) with orthogonal polynomial eigenfunctions. This allows precise asymptotic analysis of high powers of the operator. Orthogonal polynomials come in `families' and an explicit diagonalization suggests `deforming' the operator so that the full family of eigenfunctions appears. These deformations can turn out to be useful and natural. This paper gives an example with these characteristics. 

We start off with a high level description of our main results. This uses the language of Markov chains and the Burnside process. A tutorial containing background on these topics is in Sections \ref{Sect.2.1} and \ref{Sect.2.2}. See also Section \ref{Sect.4} for more details.

\subsection{The Burnside process}\label{Sect.1.1}
Let $\mathcal{X}$ be a finite set. Let $G$ be a finite group acting on $\mathcal{X}$. This splits $\mathcal{X}$ into disjoint orbits. Let $O_x=\{x^g:g\in G\}$ be the orbit containing $x$. The Burnside process gives a practical way to choose an orbit uniformly at random. This captures a familiar problem: enumerating \underline{unlabeled} objects. For example, there are $n^{n-2}$ labeled trees on $n$ vertices. There is no formula for the number of unlabeled trees. Indeed, there is an emerging literature for choosing a random spanning tree of a graph.
As a second example, let $\mathcal{X}=G$ with $G$ acting on itself by conjugation. Now the orbits are conjugacy classes of $G$. For instance, if $G=S_n$, the symmetric group, the conjugacy classes are indexed by partitions of $n$ and the Burnside algorithm gives a novel way to choose a random partition.

Let
\begin{equation*}
    \mathcal{X}_g=\{y\in\mathcal{X}: y^g=y\}, G_x=\{g\in G: x^g=x\}.
\end{equation*}
The Burnside process is a Markov chain on $\mathcal{X}$:
\begin{itemize}
    \item From $x$, choose $g\in G_x$ uniformly at random
    \item From $g$, choose $y\in \mathcal{X}_g$ uniformly at random
\end{itemize}
One step of the chain goes from $x$ to $y$. The transition matrix of this chain is
\begin{equation*}
    K(x,y)=\sum_{g\in G_x\cap G_y}\frac{1}{|G_x|}\frac{1}{|\mathcal{X}_g|},\text{ an }|\mathcal{X}|\times |\mathcal{X}|\text{ matrix}.
\end{equation*}
It is easy to see that $K$ has stationary distribution 
\begin{equation*}
    \pi(x)=\frac{z^{-1}}{|O_x|}\quad (z=\text{number of orbits}).
\end{equation*}
This means, if the chain is run from any starting state $x_0$, after `a long time' the orbit containing the current state is close to uniformly distributed.

In this generality, useful quantitative analysis of the Burnside process is an open problem. Even for the symmetric group.

We study a special case where analysis can be pushed through.

Let
\begin{itemize}
   \item $\mathcal{X}=C_2^n$ --- the binary $n$-tuples
    \item $G=S_n$ acting by permuting coordinates
\end{itemize}
If $x\in\mathcal{X}$ has $j$ ones (denoted $|x|=j$), the orbit
\begin{equation}
    O_x=\{y\in\mathcal{X}: |y|=j\}.
\end{equation}
Here $0\leq |x|\leq n$ and the Burnside process gives a complicated way of choosing an element uniformly from $\{0,1,\cdots,n\}$. It is instructive to see how the general algorithm specializes: 
\begin{itemize}
    \item For $x\in\mathcal{X}$, the permutations fixing $x$ form the subgroup $S_j\times S_{n-j}$ (for $j=|x|$) permuting the ones among themselves and the zeros among themselves. It is easy to choose $\sigma\in G_x$ at random.
    \item For $\sigma\in S_n$, write $\sigma$ as a product of cycles. Label the entries of each cycle, randomly, by a zero or a one ($2^{c(\sigma)}$ choices if $\sigma$ has $c(\sigma)$ cycles). Installing this zero\slash one pattern gives $y$. It is easy to choose $y$ given $\sigma$.
\end{itemize}
A first result of this paper shows that, for this example, a finite number of steps are necessary and sufficient for convergence no matter how large $n$ is. To say this carefully, define
\begin{equation*}
    \|K_x^l-\pi\|=\frac{1}{2}\sum_{y\in\mathcal{X}}|K^l(x,y)-\pi(y)|.
\end{equation*}

\begin{theorem}\label{Theorem1}
For $\mathcal{X}=C_2^n$, $G=S_n$, the Burnside process defined above satisfies, for all $n\geq 2$, starting state $x_0=(1,1,\cdots,1)$ and all $l\in\{1,2,3,\cdots\}$,
\begin{equation}
    \frac{1}{4}(\frac{1}{4})^l\leq \|K_{x_0}^l-\pi\|\leq 4(\frac{1}{4})^l.
\end{equation}
\end{theorem}

The development thus far doesn't seem to have much to do with orthogonal polynomials. Indeed, we do not really understand where they come from or just when they will appear in other versions of the Burnside process. Let us explain their appearance here.

Let $X_0=x,X_1,X_2,\cdots$ be the realization of the Burnside chain of Theorem \ref{Theorem1}. There is enough symmetry around that $|X_0|=|x|,|X_1|,|X_2|,\cdots$ forms a Markov chain on $\{0,1,2,\cdots,n\}$. Let $Q(i,j)$ be the transition matrix for this chain. The explicit form of $Q$ is derived in Proposition \ref{Pro} below (originally it was written out in \cite[(3.1)-(3.3)]{D}). The $Q$ chain turns out to be symmetric with a uniform stationary distribution $u(j)=\frac{1}{n+1}, 0\leq j\leq n$.  

Let $T_j^n(x)$ be the discrete Chebyshev polynomials on $\{0,\cdots,n\}$---the orthogonal polynomials for the uniform distribution (\cite{Chi}).

\begin{theorem}\label{Theorem2}
For the Markov chain $Q(i,j)$ on $\{0,\cdots,n\}$, the non-zero eigenvalues are given by $1$ and $\{\frac{\binom{2k}{k}^2}{2^{4k}}\}_{k=1}^{\lfloor\frac{n}{2}\rfloor}$. Moreover the eigenfunctions corresponding to the zero eigenvalues are the discrete Chebyshev polynomials on $\{0,1,\cdots,n\}$ with odd degree. The eigenfunction corresponding to the eigenvalue $\frac{\binom{2k}{k}^2}{2^{4k}}$ is the discrete Chebyshev polynomial on $\{0,1,\cdots,n\}$ with degree $2k$, $1\leq k\leq \frac{n}{2}$.
\end{theorem}

Theorem \ref{Theorem2} and classical analysis yield Theorem \ref{Theorem1}. The proof appears in Section \ref{Sect.3}.

\subsection{Hahn polynomials}
The discrete Chebyshev polynomials sit naturally in the family of Hahn polynomials: for $\alpha,\beta>0$, 
\begin{equation}
    Q_j(x)=~_3 F_2(-j,j+\alpha+\beta-1,-x;\alpha,-n|1)
\end{equation}
(Set $\alpha=\beta=1$ for discrete Chebyshev). $Q_j$ is a polynomial of degree $j$ for $0\leq j\leq n$. The usual notation indicates with $\alpha,\beta$ and $n$ sub and superscripts suppressed here. These polynomials are orthogonal with respect to the `beta-binomial distribution' defined and discussed in Section \ref{Sect.2.3} below. 

It is natural to ask, is there some variant of the Burnside process having $Q_j$ as eigenfunctions? This is a mathematicians question, as per the first paragraph of this introduction.
To answer it, we found a useful \underline{abstract} generalization of the general Burnside process which we hope has a life of its own. Specialized to $C_2^n$ and $S_n$ it gives a Markov chain on $\{0,1,\cdots,n\}$ with a `beta-binomial distribution' having Hahn polynomial eigenfunctions and allowing a parallel to Theorem \ref{Theorem1}. 

Section \ref{Sect.2} of this paper sets out needed background: on Markov chains, the Burnside process and orthogonal polynomials (it also gives a brief overview of the many other Markov chains with polynomial eigenfunctions---from Askey-Wilson to Macdonald!) Theorems \ref{Theorem1} and \ref{Theorem2} are proved in Section \ref{Sect.3}. Section \ref{Sect.4} explains our `twisted Burnside process' from first principles. The analogues of Theorems \ref{Theorem1} and \ref{Theorem2} are stated in Section \ref{Sect.5}. In Section 6, following a suggestion of a referee, we pass to the limit as $n$ tends to infinity to find a continuous analog of the Burnside process with Jacobi polynomial eigenfunctions. The limit is usefully related to the finite case.

\subsection{Acknowledgement}
Richard Askey's support and encouragement was crucial throughout our fledgling efforts to learn and apply the beautiful subject that he built. He took us in, patiently fielded `stupid questions', introduced us to his community and published our papers. We thank Howard Cohl and Dennis Stanton for their help. We also thank the anonymous referees for their useful comments.

\section{Background}\label{Sect.2}
This section contains background material on the Burnside process and related `auxiliary variables' algorithms (Section \ref{Sect.2.1}). Needed theory from the analytic-geometric theory of Markov chains is in Section \ref{Sect.2.2}. Background on orthogonal polynomials and connections to Markov chains (Cannings method) are in Section \ref{Sect.2.3}. Since our readership has diverse backgrounds we attempt a friendly tutorial.

\subsection{The Burnside process and auxiliary variables}\label{Sect.2.1}
The Burnside process was introduced by Mark Jerrum (\cite{J}) as a contribution to `computational P\'olya theory'. P\'olya theory is about enumeration under symmetry: ``how many different ways can ten dice be colored with red, white, blue where the order of the dice and the symmetry of a die are neglected?'' This is a standard part of combinatorics. The book by P\'olya and Read (\cite{PR}) gives a classical account.
Jerrum, working with Leslie Ann Goldberg (\cite{G,G2,J2}), managed to show that for carefully selected problems, these computations are \#P-complete (this means being polynomially equivalent to dozens of other problems that are believed to require exponential time). Our work shows that for many natural problems, the Burnside process is highly efficient. This opens up a research area---when is it good? Applications of the Burnside process to practical problems in computer science (\cite{HMV}) suggest it is worth study. 

A different motivation comes because the Burnside process is a special case of a wonderful `unifying algorithm' called variously `auxiliary variables', `data augmentation' or `hit and run'. \underline{Briefly} let $\mathcal{X}$ be a finite or countable set, $\pi(x)>0$ a probability on $\mathcal{X}$. The job is to invent a Markov chain to sample from $\pi(x)$. Auxiliary variables create `non-local' Markov chains, able to move far away in a single step. They seem to mix much more rapidly than standard local algorithms. Introduce a set $I$ of auxiliary variables. For each $i\in I$ and $x\in \mathcal{X}$, choose a proposal kernel $w_x(i)\geq 0$ (a way of moving from $x$ to $i$) such that $\sum_i w_x(i)=1$ and for each $i$, there is at least one $x$ with $w_x(i)>0$. These ingredients define a joint probability $f(x,i)=\pi(x)w_x(i)$ on $\mathcal{X}\times I$. To proceed, for each $i$, a Markov chain $K_i(x,y)$ with stationary distribution $f(x|i)$ must be specified. The auxiliary variables algorithm is
\begin{itemize}
    \item From $x$, choose $i$ with probability $w_x(i)$
    \item From $i$, choose $y$ with probability $K_i(x,y)$
\end{itemize}
The chain moves from $x$ to $y$ in one step. The transition matrix is 
\begin{equation*}
    K(x,y)=\sum_i w_x(i)K_i(x,y).
\end{equation*}
A direct calculation shows that $\pi(x)$ is a stationary distribution:
\begin{eqnarray*}
\sum_x  \pi(x)K(x,y)&=& \sum_x \pi(x)\sum_{i}w_x(i) K_i(x,y)=\sum_i \sum_x w_x(i)\pi(x) K_i(x,y)\\
&=& \sum_i m(i)\sum_x f(x|i)K_i(x,y)=\sum_i m(i)f(y|i)=\pi(y)
\end{eqnarray*}
($m(i)=\sum_x f(x,i)$ is the marginal distribution).
Of course connectedness and aperiodicity of $K$ must be checked so that the Perron-Frobenius theorem is in force. 

A comprehensive review of this class of algorithms is in \cite{AD} which gives dozens of examples and special cases including the celebrated Swendsen-Wang algorithm for sampling from the Ising model.

The point for now is that essentially \underline{none} of the algorithms have sharp running time analysis. The Burnside process falls into this class: take $I=G$, $w_x(i)$ the uniform distribution on $G_x$ and $K_g(x,y)$ the uniform distribution on $\mathcal{X}_g$. The hope is, because of the extra group structure available, analysis will be easier for this special case. That was our initial motivation.

There has been some previous effort to study the Burnside process on $\mathcal{X}=[k]^n$, with $G=S_n$---here $[k]=\{1,2,\cdots,k\}$. See \cite{D} and Example \ref{Example1} in Section \ref{Sect.4} for the description of this process.

Theorems \ref{Theorem1} and \ref{Theorem2} study $k=2$. Jerrum (\cite{J}) showed that for $k=2$ the total variation mixing time is at most of order $\sqrt{n}$; sharpening this, Diaconis (\cite{D}) proved that, for $k=2$ and $x_0=(1,\cdots,1)$, $\|K_{x_0}^l-\pi\|\leq (1-c)^l$ with $c=\frac{1}{\pi}$ (with $\pi=3.14159\cdots$). Theorem \ref{Theorem1} sharpens this to get the exact rate and makes the connection to orthogonal polynomials (versus the complex, difficult argument of \cite{D}). David Aldous (\cite{AF}) used a clever coupling argument to show that, for general $k$, $\|K_x^l-\pi\|\leq n(1-\frac{1}{k})^l$. This implies $l$ of order $\log{n}$ steps suffice for $k=2$. Further results are in \cite{Chen}.

\subsection{Analysis of Markov chain convergence}\label{Sect.2.2}
For readers unfamiliar with Markov chains, the best introduction is the book by David Levin and Yuval Peres \cite{LP}. This contains all the bounds below and much more. The \underline{analytic-geometric} theory of Markov chains is well developed in Laurent Saloff-Coste's \cite{Sal}.

Let $\mathcal{X}$ be a finite set, $\pi(x)>0,\sum_x\pi(x)=1$ a probability distribution on $\mathcal{X}$. A Markov chain on $\mathcal{X}$ is specified by a matrix $K(x,y)\geq 0$ with $\sum_y K(x,y)=1$. Suppose that $\pi$ is a stationary distribution for $K$: $\sum_x \pi(x)K(x,y)=\pi(y)$ (so $\pi$ is a left eigenvector for $K$ with eigenvalue $1$). Call $\pi,K$ \underline{reversible} if the detailed balance condition $\pi(x)K(x,y)=\pi(y)K(y,x)$ is satisfied for all $x,y\in\mathcal{X}$. All the Markov chains below satisfy detailed balance. Let $L^2(\pi)=\{f:\mathcal{X}\rightarrow \mathbb{R}: \sum f(x)^2\pi(x)<\infty  \}$. $K$ acts on $L^2(\pi)$ by $Kf(x)=\sum_y K(x,y)f(y)$. Detailed balance is equivalent to saying that $\langle Kf|g \rangle=\langle f|Kg \rangle$ so $K$ is a bounded, self-adjoint operator on $L^2(\pi)$. The spectral theorem is in force: there exist \underline{eigenvalues} $\beta_i$ ($1=\beta_0\geq \beta_1\geq \cdots \geq \beta_{|\mathcal{X}|-1}\geq -1$) and \underline{eigenfunctions} $\psi_i(x)$ (so $K\psi_i(x)=\beta_i\psi_i(x)$). 

Throughout this section assume $K$ is connected (for every $x,y\in\mathcal{X}$ there is $l\geq 1$ with $K^l(x,y)>0$) and aperiodic ($\beta_{|\mathcal{X}|-1}>-1$). It is useful to introduce two distances
\begin{itemize}
    \item Total variation distance---$\|K_x^l-\pi\|=\frac{1}{2}\sum_y|K^l(x,y)-\pi(y)|\\
    =\max_{\|f\|_{\infty}\leq 1} K^l(f)-\pi(f)$. Here $\|f\|_{\infty}=\max_x |f(x)|,\pi(f)=\sum \pi(x)f(x)$. This equality is easy to prove (the maximizing $f$ is $1$ at $y$ if $K^l(x,y)\geq \pi(y)$ and $-1$ otherwise). 
    \item Chi-square distance---$\chi_{x}^2(l)=\sum_y \frac{(K^l(x,y)-\pi(y))^2}{\pi(y)}$.
\end{itemize}
The Cauchy-Schwarz inequality implies
\begin{equation}
    4\|K_x^l-\pi\|^2\leq \chi_x^2(l).
\end{equation}
This is a useful route to getting bounds on convergence: bound $L^1$ by $L^2$ and use eigenvalues to bound $L^2$. Eigenvalues and eigenfunctions come in by
\begin{equation}
    \chi_x^2(l)=\sum_{i=1}^{|\mathcal{X}|-1}\beta_i^{2l}\psi_i^2(x).
\end{equation}

The following lower bound on total variation and chi-square distance is proved in \cite[Lemma 2.1]{DKS}. It is applied in Section \ref{Sect.3}.

\begin{proposition}\label{Lower}
Suppose that $\psi$ is an eigenfunction of $K$ with eigenvalue $\beta$ such that $\pi(\psi)=0$. Assume $\|\psi\|^2=1$. Then
\begin{equation*}
    \chi^2_x(l)\geq |\psi(x)|^2|\beta|^{2l}.
\end{equation*}
For $\psi$ with $\|\psi\|_{\infty}<\infty$, \begin{equation*}
    \|K_x^l-\pi\|\geq \frac{|\psi(x)||\beta|^l}{2\|\psi\|_{\infty}}.
\end{equation*}
\end{proposition}

\subsection{Orthogonal polynomials and Markov chains}\label{Sect.2.3}
Background on orthogonal polynomials is wonderfully presented in the introductory text by Chihara \cite{Chi}. The recent encyclopedia by Ismail \cite{Ism} brings in much further material. 

\paragraph{Hahn polynomials}
These are the orthogonal polynomials on $\{0,1,\cdots,n\}$ with respect to the beta-binomial distribution. For reference and variations of the beta-binomial distribution see \cite{Wil}. Fix $\alpha,\beta>0$. Define 
\begin{equation}\label{Beta}
    m(j)=\binom{n}{j}\frac{(\alpha)_j(\beta)_{n-j}}{(\alpha+\beta)_{n}}, (\alpha)_x=\alpha(\alpha+1)\cdots (\alpha+x-1), (\alpha)_0=1.
\end{equation}
This is a probability distribution on $\{0,1,\cdots,n\}$ generated as a beta mixture of binomial distributions
\begin{equation*}
    m(j)=\int_{0}^1 \binom{n}{j}x^j(1-x)^{n-j}\frac{\Gamma(\alpha+\beta)}{\Gamma(\alpha)\Gamma(\beta)}x^{\alpha-1}(1-x)^{\beta-1}dx.
\end{equation*}
When $\alpha=\beta=1$, $m(j)=\frac{1}{n+1}$---the uniform distribution.

The orthogonal polynomials for $m(j)$ are called \underline{Hahn polynomials}. Detailed development is in \cite{KM}.
They have the explicit form
\begin{equation*}
     Q_j(x)=~_3F_2(-j,j+\alpha+\beta-1,-x;\alpha,-n|1)
\end{equation*}
normalized so $Q_j(0)=1$, $Q_j(n)=\frac{(-\beta-j)_j}{(\alpha+1)_j}$.
The standard form
\begin{equation*}
    {}_r F_s(a_1,\cdots,a_r;b_1,\cdots,b_s|x)=\sum_{l=0}^{\infty}\frac{(a_1\cdots a_r)_l x^l}{(b_1\cdots b_s)_l l!}
\end{equation*}
with $(a_1\cdots a_r)_l=\prod_{i=1}^r(a_i)_l$, is used. When $\alpha=\beta=1$ these are the discrete Chebyshev polynomials (which we denote by $T_j(x)$)
\begin{equation*}
    T_0(x)=1, T_1(x)=\frac{n-2x}{n},T_2(x)= \frac{6x^2-6nx+n(n-1)}{n(n-1)}.
\end{equation*}
The discrete Chebyshev polynomials satisfy the following recurrence relation (\cite[(6.2.8)]{Ism})
\begin{equation}\label{Recurrence}
    (j+1)(n-j) T_{j+1}(x)=(2j+1)(n-2x)T_j(x)-j(j+n+1)T_{j-1}(x),
\end{equation}
where $j=0,1,\cdots,n-1$ and $T_{-1}(x)=0$.

In Section \ref{Sect.3} the proof of Theorem \ref{Theorem2} uses `Cannings argument' (\cite{Can}), a frequently used tool for proving that a Markov chain has orthogonal polynomial eigenfunctions:

`If the operator $K$ sends polynomials of degree $j$ to polynomials of degree $j$ for all $0\leq j\leq n$, then the operator has the orthogonal polynomials for the stationary distribution as its eigenfunctions.'

The following lemma is a formal statement and proof of Cannings lemma. A very similar argument works for multivariate polynomials.

\begin{lemma}\label{L1.1_4}
Suppose that $K_n$ is the transition matrix of a reversible Markov chain with state space $\mathcal{X}=\{0,1,\cdots,n\}$ and stationary distribution $\pi$. Suppose that for any polynomial $f$ on $\mathcal{X}$ of degree $l\leq n$, $K_nf$ is a polynomial on $\mathcal{X}$ of degree less than or equal to $l$. Then we conclude that the eigenfunctions of $K_n$ are given by orthogonal polynomials for $\pi$ of degree less than or equal to $n$. 

Moreover, for any $l\in\mathbb{N}$, if we further assume that the degree $l$ monomial of $K_n[x^l]$ does not depend on $n$ for $n\geq l$, then the eigenvalue of $K_n$ that corresponds to the eigenfunction of degree $l$ does not depend on $n$ for $n\geq l$.
\end{lemma}
\begin{proof}
We prove by induction that for any $0\leq l\leq n$, the orthogonal polynomial for $\pi$ of degree $l$ is an eigenfunction of $K_n$. We denote the $l$th orthogonal polynomial for $\pi$ by $\phi_{n,l}$. For $l=0$, we have $K_n \vec{1}=\vec{1}$, where $\vec{1}$ is the constant function taking value $1$. Now for $l$ such that $1\leq l\leq n$, by the induction hypothesis we suppose that for any $0\leq i\leq l-1$, $K_n\phi_{n,i}=\lambda_{n,i}\phi_{n,i}$. Now $\phi_{n,l}$ is a polynomial of degree $l$, hence $K_n\phi_{n,l}$ is a polynomial of degree $\leq l$. Expanding $K_n\phi_{n,l}$ in terms of $\{\phi_{n,i}\}_{i=0}^l$, we get
\begin{equation}\label{eigenv}
    K_n\phi_{n,l}=\lambda_{n,l}\phi_{n,l}+\sum_{i=0}^{l-1}a_i\phi_{n,i}
\end{equation}
for some $\lambda_{n,l}$, $\{a_i\}_{i=0}^{l-1}$. Now consider the inner product for $\pi$ defined as
\begin{equation*}
    \langle f,g\rangle_{\pi}:=\sum_{k=0}^n f(k)g(k)\pi(k).
\end{equation*}
$K_n$ is self-adjoint with respect to this inner product, hence we have for any $0\leq i\leq l-1$,
\begin{equation*}
 \langle K_n\phi_{n,l},\phi_{n,i} \rangle_{\pi}=\langle \phi_{n,l},K_n\phi_{n,i} \rangle_{\pi},
\end{equation*}
which leads to
\begin{equation*}
    a_i=\frac{\lambda_{n,i}\langle\phi_{n,l},\phi_{n,i}\rangle_{\pi}}{\langle\phi_{n,i},\phi_{n,i}\rangle_{\pi}}=0.
\end{equation*}
Therefore we conclude that $K_n\phi_{n,l}=\lambda_{n,l}\phi_{n,l}$.

Now we assume that the degree $l$ monomial of $K_n[x^l]$ does not depend on $n$ for $n\geq l$. By (\ref{eigenv}), the degree $l$ monomial of $K_n[x^l]$ is given by $\lambda_{n,l} x^l$, hence $\lambda_{n,l}$ does not depend on $n$ for $n\geq l$. Note that from the argument in the last paragraph, the eigenvalue of $K_n$ that corresponds to the eigenfunction $\phi_{n,l}$ is given by $\lambda_{n,l}$, which does not depend on $n$ for $n\geq l$.
\end{proof}

The connections between orthogonal polynomials and Markov chains is long-standing. Indeed, perhaps the first Markov chain---the Bernoulli-Laplace urn---is diagonalized by Hahn polynomials. To briefly recall, Bernoulli and Laplace considered two urns, the left containing $n$ red balls, the right containing $n$ black balls. Each time, a ball is picked uniformly at random and the two balls are switched. The number of red balls in the left urn evolves as a Markov chain on $\{0,1,\cdots,n\}$. In \cite{DS} this chain is diagonalized and shown to have Hahn polynomial eigenfunctions. That same paper treats the Ehrenfest's urn (with Krawtchouk polynomial eigenfunctions) and several $q$-deformations of these two examples. These urn models, and the birth and death chains discussed below, are examples of `local Markov chains'. Most appearances of orthogonal polynomials have occurred for such local chains (analogs of the Laplacian). The present examples are much more vigorous. 

Jacobi polynomials arise as eigenfunctions of a Markov chain constructed from the `Gibbs sampler' \cite{DKS}. Extensions of this construction need the full power of the Askey-Wilson polynomials (\cite{BW}).

An extensive connection between orthogonal polynomials and birth-death chains follows from the Karlin-McGregor theory. A textbook account with full details is in \cite{And}.
Multivariate orthogonal polynomials as in \cite{DX} arise in natural genetics problems. See \cite{DiaGri,KZ,Xu} for multivariate Hahn and Krawtchouk polynomials and \cite{DR} for Macdonald polynomials. This is just a small sample, drawn from our work with colleagues. The list goes on and on.

\section{Proof of Theorems \ref{Theorem1} and \ref{Theorem2}}\label{Sect.3}
The proof consists of three main parts. First we show that the operator $Q$ of Section \ref{Sect.1.1} sends polynomials to polynomials. Thus Cannings lemma (Lemma \ref{L1.1_4}) shows that the discrete Chebyshev polynomials are the eigenfunctions of the Burnside process (lumped to orbits). Next, the eigenvalues are computed, proving Theorem \ref{Theorem2}. Finally, the eigen structure and analytic tools of Section \ref{Sect.2.2} are used to prove Theorem \ref{Theorem1}.

Throughout this section $\alpha=\beta=1$ and the stationary distribution $u(j)=\frac{1}{n+1}$ for $0\leq j\leq n$.

\subsection{Proof of Theorem \ref{Theorem2}}

In the following, we denote by $T_{j}^n(x)$ the discrete Chebyshev polynomial on $\{0,1,\cdots,n\}$ with degree $j$ (to emphasize the dependence on $n$). We also assume that $n$ is even below (the proof for odd $n$ is similar).

Before the proof of Theorem \ref{Theorem2}, we present some preparatory lemmas.

The following lemma follows from P\'olya's cycle index theorem (\cite{SL},\cite[Section 5]{DS2}).
We recall the notation. For $\sigma$ in the symmetric group $S_n$, write $a_i(\sigma)$ for the number of $i$-cycles when $\sigma$ is written in cycle notation. So $a_1(\sigma)$ is the number of fixed points, $a_2(\sigma)$ is the number of transpositions,$\cdots$ Thus $0\leq a_i(\sigma)\leq n$ and $\sum_{i=1}^n i a_i(\sigma)=n$. Write the cycle index of $S_n$ as 
\begin{equation}\label{Pol1}
    C_n(x_1,\cdots,x_n)=\frac{1}{n!}\sum_{\sigma\in S_n}\prod_{i=1}^{n}  x_i^{a_i(\sigma)}.
\end{equation}
The generating function of these $C_n$ is $C(t)=\sum_{n=0}^{\infty}C_n t^n$. P\'olya showed
\begin{equation}\label{Pol2}
    C(t)=e^{\sum_{i=1}^{\infty}x_i \frac{t^i}{i}}.
\end{equation}
Repeatedly differentiating in the $x_i$, setting $x_i=1$ and comparing coefficient gives 

\begin{lemma}\label{L1.1_6}
For $n=1,2,\cdots$, any $1\leq k_1<\cdots<k_r\leq n$ and $l_1,\cdots,l_r\geq 1$, if $\sum_{i=1}^r k_i l_i\leq n$, then for a random permutation $\sigma$ uniformly chosen in $S_n$, 
\begin{equation}
    \prod_{i=1}^r(l_i)!\mathbb{E}\Big[\prod_{i=1}^r\binom{a_{k_i}(\sigma)}{l_i}\Big]=\prod_{i=1}^r\frac{1}{k_i^{l_i}};
\end{equation}
while if $\sum_{i=1}^r k_i l_i> n$, then
\begin{equation}
    \mathbb{E}\Big[\prod_{i=1}^r\binom{a_{k_i}(\sigma)}{l_i}\Big]=0.
\end{equation}
\end{lemma}

Below, by ``even order terms'' we mean monomials of the form $\alpha X_{i_1}^{l_1}\cdots X_{i_k}^{l_k}$ with $i_1<\cdots<i_k$, $\alpha\in\mathbb{R}$ and $l_1,\cdots,l_k$ even.
\begin{lemma}\label{L1.1_5}
Suppose that $N,a\geq 1$. The ``even order terms'' in the expansion of $(\sum_{i=1}^N X_i)^{2a}$ can be expressed as a linear combination of the following form
\begin{equation}\label{Eq33}
    \prod_{q=1}^k(\sum_{i=1}^N X_i^{2l_q})
\end{equation}
for some $k\geq 1$, $l_1,\cdots,l_k\geq 1$ and $\sum_{q=1}^k l_q=a$. Moreover, the coefficients in the linear combination only depend on $a$.
\end{lemma}
\begin{proof}
By adding finitely many indeterminates and taking them to be $0$ in the end, we can assume without loss of generality that $N\geq 2a$.
We note that the ``even order terms'' of $(\sum_{i=1}^N X_i)^{2a}$ is a symmetric polynomial in the indeterminates $X_1^2,\cdots,X_N^2$. By the proof of \cite[Proposition 2.9]{BO}, it can be expressed as linear combination of the form (\ref{Eq33}), and the coefficients only depend on $a$.
\end{proof}

\begin{proof}[Proof of Theorem \ref{Theorem2}, eigenfunction part]
Note that $T_0^n(x)=1$ and $T_1^n(x)=\frac{n-2x}{n}$. From this and the recurrence relation (\ref{Recurrence}) it can be derived that 
for $0\leq k\leq \frac{n}{2}$, $T^n_{2k}$ only has even-order terms in $(x-\frac{n}{2})$, and $T^n_{2k-1}$ (for $k\geq 1$) only has odd-order terms in $(x-\frac{n}{2})$.

Now as $Q(i,j)=Q(i,n-j)$, we conclude that $Q[(x-\frac{n}{2})^{2k-1}]=0$ for any $1\leq k\leq \frac{n}{2}$. By the previous conclusion, $Q T_{2k-1}^n=0$. 
Hence $\{T_{2k-1}^n\}_{k=1}^{\frac{n}{2}}$ are eigenfunctions of $K$ corresponding to the zero eigenvalues.

In view of Lemma \ref{L1.1_4}, in order to show that the rest of the eigenfunctions are also given by the discrete Chebyshev polynomials, it suffices to show that for any $a$ such that $1\leq a\leq \frac{n}{2}$, $Q[(x-\frac{n}{2})^{2a}](j)$ is a polynomial in $j$ of degree $\leq 2a$ (where $j\in \{0,1,\cdots,n\}$). Below we verify this fact.

Suppose that we start from the orbit $j$. For the first step in the Burnside process, $\sigma_1\in S_j$ and $\sigma_2\in S_{n-j}$ are drawn uniformly. Let $\{a_i\}_{i=1}^n$, $\{b_i\}_{i=1}^n$ denote the number of cycles of length $i$ in $\sigma_1,\sigma_2$, respectively. For the second step, we label the entries of each cycle by $0\slash 1$ independently. Let $Z_i\sim Binomial(a_i+b_i,\frac{1}{2})$. The outcome of one iteration of the Burnside process (number of coordinates labeled by $1$) can be represented by $X=\sum_{i=1}^n iZ_i$. For $1 \leq a\leq \frac{n}{2}$, we denote by
\begin{equation}
    W_{2a}:=\mathbb{E}\Big[(X-\frac{1}{2}n)^{2a}\Big]=\mathbb{E}\Big[(X-\sum_{i=1}^n\frac{1}{2}i(a_i+b_i))^{2a}\Big].
\end{equation}
We note that
\begin{equation}
    W_{2a}=\mathbb{E}\Big[\mathbb{E}\Big[(X-\sum_{i=1}^n\frac{1}{2}i(a_i+b_i))^{2a}|\sigma_1,\sigma_2\Big]\Big].
\end{equation}
Below we denote by $\tilde{W}_{2a}(\sigma_1,\sigma_2)$ the conditional expectation inside the above expression.

By the multinomial theorem and vanishing of odd moments, we obtain that
\begin{eqnarray}\label{Eq1.1}
\tilde{W}_{2a}(\sigma_1,\sigma_2)&=&\mathbb{E}\Big[(\sum_{i=1}^n i(Z_i-\frac{1}{2}(a_i+b_i)))^{2a}|\sigma_1,\sigma_2\Big]\nonumber \\
&=&\sum_{\substack{j_1+\cdots+j_n=a,\\j_1,\cdots,j_n\geq 0}}\binom{2a}{2j_1,\cdots,2j_n}\nonumber\\
&\times&\prod_{i=1}^n i^{2j_i}\mathbb{E}\Big[(Z_i-\frac{1}{2}(a_i+b_i))^{2 j_i}|\sigma_1,\sigma_2\Big].
\end{eqnarray}
By splitting $Z_i$ into Bernoulli random variables and using the multinomial theorem again we obtain
\begin{equation}
    \mathbb{E}\Big[(Z_i-\frac{1}{2}(a_i+b_i))^{2 j_i}|\sigma_1,\sigma_2\Big]=\frac{1}{2^{2j_i}}\sum_{\substack{p_{i,1}+\cdots+p_{i,a_i+b_i}=j_i\\p_{i,1},\cdots,p_{i,a_i+b_i}\geq 0}}\binom{2j_i}{2p_{i,1},\cdots,2p_{i,a_i+b_i}}.
\end{equation}
Plugging the above expression into (\ref{Eq1.1}), we get
\begin{eqnarray}\label{Eq1.2}
\tilde{W}_{2a}(\sigma_1,\sigma_2)&=&\frac{1}{2^{2a}}\sum_{\substack{\sum\limits_{i=1}^n\sum\limits_{q=1}^{a_i+b_i}p_{i,q}=a\\p_{i,q}\geq 0}}
\prod_{i=1}^n i^{2(p_{i,1}+\cdots+p_{i,a_i+b_i})}
\nonumber\\
&\times& \binom{2a}{2p_{1,1},\cdots,2p_{1,a_1+b_1},\cdots,2p_{n,1},\cdots,2p_{n,a_n+b_n}}.
\end{eqnarray}

We note that by (\ref{Eq1.2}), $2^{2a}\tilde{W}_{2a}$ can be obtained from ``even order terms'' in 
\begin{equation}
    (\sum_{i=1}^n\sum_{q=1}^{a_i+b_i} X_{i,q})^{2a}
\end{equation}
by substituting $X_{i,q}=i$ for $1\leq q\leq a_i+b_i$. Now by lemma \ref{L1.1_5}, the ``even order terms'' of $(\sum_{i=1}^n\sum_{q=1}^{a_i+b_i} X_{i,q})^{2a}$ can be expressed in terms of the following forms
\begin{equation}\label{Eq4}
    \prod_{r=1}^k(\sum_{i=1}^n\sum_{q=1}^{a_i+b_i}X_{i,q}^{2l_r}),
\end{equation}
with $k\geq 1$, $l_1,\cdots,l_k\geq 1$ and $\sum_{r=1}^k l_r=a$. Moreover, the coefficients in the linear combination only depend on $a$. 
By substituting $X_{i,q}=i$ and taking expectation in (\ref{Eq4}), we get
\begin{equation}
    \mathbb{E}\Big[\prod_{r=1}^k(\sum_{i=1}^n i^{2l_r}(a_i+b_i))\Big].
\end{equation} 

In order to make use of Lemma \ref{L1.1_6}, we introduce the following subspaces. We take the base field to be $\mathbb{R}$. For every monomial $\alpha X_{i_1}^{d_1}\cdots X_{i_k}^{d_k}$ ($i_1<\cdots<i_k$,$\alpha\in \mathbb{R}$), we define 
\begin{equation}
    B(\alpha X_{i_1}^{d_1}\cdots X_{i_k}^{d_k}):=\alpha(d_1)!\cdots (d_k)!\binom{X_{i_1}}{d_1}\cdots\binom{X_{i_k}}{d_k},
\end{equation}
and extend linearly. We also denote by
\begin{equation}
    X(l_1,\cdots,l_k)=\prod_{r=1}^k(\sum_{i=1}^n i^{2l_r}X_i),
\end{equation}
and $B(l_1,\cdots,l_r):=B(X(l_1,\cdots,l_r))$. For any $d\geq 1$, we define $\mathcal{X}(d,a)$ to be the set of $X(l_1,\cdots,l_k)$ with $1\leq k\leq d$ and $l_1+\cdots+l_k=a$, and $\mathcal{B}(d,a)$ to be the set of $B(l_1,\cdots,l_k)$ with $1\leq k\leq d$ and $l_1+\cdots+l_k=a$. We also define $\tilde{\mathcal{X}}(d,a)$ to be the linear span of $\mathcal{X}(d,a)$ and $\tilde{\mathcal{B}}(d,a)$ to be the linear span of $\mathcal{B}(d,a)$.

Now we prove that for $l_1+\cdots+l_k=a$, $\tilde{\mathcal{X}}(k,a)=\tilde{\mathcal{B}}(k,a)$, and that it is possible to express $X(l_1,\cdots,l_k)$ in terms of elements of $\mathcal{B}(k,a)$ so that the coefficients do not depend on $n$. The method is to prove the following stronger claim by induction on $k$: for any $l_1+\cdots+l_k=a$, $\tilde{\mathcal{X}}(k,a)=\tilde{\mathcal{B}}(k,a)$, $X(l_1,\cdots,l_k)-B(l_1,\cdots,l_k)\in \tilde{\mathcal{X}}(k-1,a)$, and it is possible to express $X(l_1,\cdots,l_{k})$ in terms of elements of $\mathcal{B}(k,a)$, $B(l_1,\cdots,l_k)$ in terms of elements of $\mathcal{X}(k,a)$ and $X(l_1,\cdots,l_k)-B(l_1,\cdots,l_k)$ in terms of elements of $\mathcal{X}(k-1,a)$ so that the coefficients do not depend on $n$. When $k=1$, $X(l_1)=B(l_1)$ for any $l_1$, and the result holds. For the induction step, we assume that the claim holds for $\leq k$, and consider the $k+1$ case. We will make use of the following identity:
\begin{eqnarray*}
&&X(l_1,\cdots,l_{k+1})-B(l_1,\cdots,l_{k+1})\\
&=&\frac{1}{k+1}\sum_{j=1}^{k+1}(X(l_1,\cdots,\hat{l_j},\cdots,l_{k+1})-B(l_1,\cdots,\hat{l_j},\cdots,l_{k+1}))\\
&&\times (\sum_{i=1}^n i^{2l_j}X_i)\\
&&+\frac{1}{k+1}\sum_{j=1}^{k+1}\sum_{j'\neq j}B(l_1,\cdots,l_{j'}+l_j,\cdots,\hat{l_j},\cdots,l_{k+1}).
\end{eqnarray*}
The identity can be proved by matching individual terms on both sides. Now by the induction hypothesis, when $l_1+\cdots+l_{k+1}=a$, $X(l_1,\cdots,\hat{l_j},\cdots,l_{k+1})-B(l_1,\cdots,\hat{l_j},\cdots,l_{k+1})$ is in $\tilde{\mathcal{X}}(k-1,a-l_j)$. Therefore the first term of the right hand side is in $\tilde{\mathcal{X}}(k,a)$. The second term of the right hand side is in $\tilde{\mathcal{B}}(k,a)$, hence in $\tilde{\mathcal{X}}(k,a)$. Thus the left hand side is in $\tilde{\mathcal{X}}(k,a)=\tilde{\mathcal{B}}(k,a)$. Hence $\tilde{\mathcal{X}}(k+1,a)=\tilde{\mathcal{B}}(k+1,a)$. The last conclusion can also be checked from the identity. 

Now $\mathbb{E}\Big[\prod_{r=1}^k(\sum_{i=1}^n i^{2l_r}(a_i+b_i))\Big]$ can be expanded in terms of the forms (coefficients do not depend on $n,j$)
\begin{equation}
\mathbb{E}\Big[\prod_{r\in \Gamma}(\sum_{i=1}^n i^{2l_r}a_i)\Big]\mathbb{E}\Big[\prod_{r\in [k]\backslash\Gamma}(\sum_{i=1}^n i^{2l_r}b_i)\Big]
\end{equation}
where $\Gamma\subseteq [k]$. By the above result, the first factor can be expressed in terms of elements in $\mathcal{B}(|\Gamma|,\sum_{r\in\Gamma} l_r)$, and similarly for the second factor. 

Now we denote by $B'(l_1,\cdots,l_k)$ the result obtained by substituting $X_i=a_i$ in $B(l_1,\cdots,l_k)$, and $B''(l_1,\cdots,l_k)$ similar with $X_i=b_i$. Note that by Lemma \ref{L1.1_6},
\begin{equation}
    \mathbb{E}[B'(l_1,\cdots,l_k)]=\sum_{i_1+\cdots+i_k\leq j}i_1^{2l_1-1}\cdots i_k^{2l_k-1}.
\end{equation}
It can be shown by induction on $k$ that the above expression is a polynomial in $j$ with degree $2(l_1+\cdots+l_k)$. Similarly, it can be shown that $\mathbb{E}[B''(l_1,\cdots,l_k)]$ is a polynomial in $(n-j)$ of degree $2(l_1+\cdots+l_k)$, hence it is also a polynomial of the same degree in $j$ with coefficients depending on $n$ (but the coefficient of highest degree does not depend on $n$). Putting these together, we conclude that $\mathbb{E}\Big[\prod_{r=1}^k(\sum_{i=1}^n i^{2l_r}(a_i+b_i))\Big]$ is a polynomial of degree $\leq 2a$ in $j$, and the coefficient of its degree $2a$ term does not depend on $n$. 

We conclude that $W_{2a}$ is a polynomial of degree $\leq  2a$ in $j$, and its coefficient of degree $2a$ does not depend on $n$ for $n\geq 2a$. This implies that the coefficient of degree $2a$ of $\mathbb{E}[X^{2a}]$ does not depend on $n$ for $n\geq 2a$. By Lemma \ref{L1.1_4}, the eigenfunctions of $K$ are given by the discrete Chebyshev polynomials, and the eigenvalue corresponding to $T^n_{2k}$ does not depend on $n$ as long as $n\geq 2k$.
\end{proof}

Now we present the proof for the eigenvalues. We make use of three lemmas in the proof, which we also present below. The first gives needed inequalities for the gamma function. For many further references, see \cite{Gor}. 

\begin{lemma}[Explicit Stirling approximation]\label{L1.1_1}
For any $x>0$,
\begin{equation}
    \sqrt{\pi}(\frac{x}{e})^x(8x^3+4x^2+x+\frac{1}{100})^{\frac{1}{6}}<\Gamma(1+x)<\sqrt{\pi}(\frac{x}{e})^x(8x^3+4x^2+x+\frac{1}{30})^{\frac{1}{6}}.
\end{equation}
From this it can be easily derived
that for any $x\geq 1$,
\begin{equation}
    \sqrt{2\pi}(\frac{x}{e})^x\sqrt{x}<\Gamma(1+x)<\sqrt{\pi}(\frac{x}{e})^x\sqrt{2x}(1+\frac{1}{x})^{\frac{1}{6}}.
\end{equation}
\end{lemma}

\begin{lemma}[Clausen \cite{Cla}, see also \cite{M}]\label{L1.1_2}
\begin{equation}
    ~_3F_2(2a,2b,a+b; a+b+\frac{1}{2},2a+2b|x)=({}_2F_1(a,b;a+b+\frac{1}{2}|x))^2.
\end{equation}
\end{lemma}

\begin{lemma}[Gauss's hypergeometric theorem, see Page 2 of \cite{Bailey}]\label{L1.1_3}
For $a,b,c\in\mathbb{R}$ such that $a+b<c$, we have
\begin{equation}
    {}_2F_1(a,b;c|1)=\frac{\Gamma(c)\Gamma(c-a-b)}{\Gamma(c-a)\Gamma(c-b)}.
\end{equation}
\end{lemma}

\begin{proof}[Proof of Theorem \ref{Theorem2}, eigenvalue part]
We use the notation as in the statement of Theorem \ref{Theorem2}. Let $\lambda_{k,n}$ be the eigenvalue corresponding to the eigenfunction $T_{2k}^n$ for the Markov chain $Q(i,j)$ on $\{0,1,\cdots,n\}$. Note that from the determination of eigenfunctions, $\lambda_{k,n}$ is constant for all $n\geq 2k$, and we denote it by $\lambda_k$. The strategy of the proof is to analyze the expression for $\lambda_k$ in a system of size $n$ and let $n\rightarrow \infty$ to get the desired result.

The eigenfunction that corresponds to $\lambda_k$ has been proven previously to be the discrete Chebyshev polynomials $T_{2k}^n$. Thus by examining the first row of the eigenvalue-eigenfunction equation, we obtain that (see Proposition \ref{Pro} and \cite[(3.1)-(3.3)]{D})
\begin{equation}
    \lambda_{k,n}=\frac{\sum_{j=0}^n \alpha_{j}^n T_{2k}^n(j)}{T^n_{2k}(0)}=\sum_{j=0}^n \alpha_{j}^n T_{2k}^n(j),
\end{equation}
where $\alpha_j^n=\frac{\binom{2j}{j}\binom{2(n-j)}{n-j}}{2^{2n}}$.

Plugging the expression
\begin{equation}
    T_{2k}^n(j)=\sum_{l=0}^{2k}\frac{(-2k)_l(2k+1)_l(-j)_l}{(l!)^2(-n)_l}
\end{equation}
into the expression above, we obtain that
\begin{equation}\label{Eq1.1_11}
    \lambda_{k}=\sum_{l=0}^{2k}\frac{(-2k)_l(2k+1)_l}{ (l!)^2}\sum_{j=0}^n\frac{(-j)_l\alpha_{j}^n}{(-n)_l}.
\end{equation}

Below we prove that 
\begin{equation}\label{Eq1.1_1}
    \lim_{n\rightarrow\infty}\sum_{j=0}^n\frac{(-j)_l\alpha^n_j}{(-n)_l}=\frac{(\frac{1}{2})_l}{l!}.
\end{equation}
We fix a small $c>0$, and assume that $n$ is sufficiently large below. Note that by Lemma \ref{L1.1_1}, we can derive that for $n$ sufficiently large,
\begin{equation}
|\sum_{j\leq cn}\frac{(-j)_l\alpha^n_j}{(-n)_l}|\leq \sum_{j\leq cn}\frac{\binom{2j}{j}\binom{2(n-j)}{n-j}}{2^{2n}}\leq \frac{2}{\sqrt{n}}+\sum_{1\leq j\leq cn}\frac{4}{\sqrt{j}\sqrt{n-j}}\leq \frac{10\sqrt{c}}{\sqrt{1-c}}.
\end{equation}
Similarly,
\begin{equation}
    |\sum_{(1-c)n \leq j\leq n}\frac{(-j)_l\alpha^n_j}{(-n)_l}|\leq \frac{10\sqrt{c}}{\sqrt{1-c}}.
\end{equation}
Now note that for $n$ sufficiently large, when $cn < j<(1-c)n$, we have $(\frac{j}{n})^l(1-\frac{l+1}{cn})^l\leq \frac{(-j)_l}{(-n)_l} \leq (\frac{j}{n})^l$. Moreover, using Lemma \ref{L1.1_1}, we obtain that $\frac{1}{\pi}\frac{1}{\sqrt{j}\sqrt{n-j}}(1+\frac{6}{cn})^{-\frac{2}{3}}\leq\alpha^n_j\leq \frac{1}{\pi}\frac{1}{\sqrt{j}\sqrt{n-j}}(1+\frac{6}{c n})^{\frac{1}{3}}$ for $cn <j<(1-c)n$. Therefore we have
\begin{equation}
    \sum_{cn<j<(1-c)n}\frac{(-j)_l\alpha^n_j}{(-n)_l}\leq (1+\frac{6}{cn})^{\frac{1}{3}}\frac{1}{\pi}\sum_{cn <j<(1-c)n}\frac{1}{\sqrt{\frac{j}{n}}\sqrt{1-\frac{j}{n}}}(\frac{j}{n})^l\frac{1}{n},
\end{equation}
\begin{equation}
    \sum_{cn<j<(1-c)n}\frac{(-j)_l\alpha^n_j}{(-n)_l}\geq (1+\frac{6}{cn})^{-\frac{2}{3}}(1-\frac{l+1}{cn})^{l}\frac{1}{\pi}\sum_{cn <j <(1-c)n}\frac{1}{\sqrt{\frac{j}{n}}\sqrt{1-\frac{j}{n}}}(\frac{j}{n})^l\frac{1}{n}.
\end{equation}
Now note that
\begin{equation}
    \lim_{n\rightarrow\infty}\frac{1}{\pi}\sum_{cn<j<(1-c)n}\frac{1}{\sqrt{\frac{j}{n}}\sqrt{1-\frac{j}{n}}}(\frac{j}{n})^l\frac{1}{n}=\frac{1}{\pi}\int_{c}^{1-c}\frac{x^l}{\sqrt{x}\sqrt{1-x}}.
\end{equation}
Hence we have
\begin{eqnarray*}
    &&\frac{1}{\pi}\int_{c}^{1-c}\frac{x^l}{\sqrt{x}\sqrt{1-x}}-\frac{20\sqrt{c}}{\sqrt{1-c}}\leq \liminf_{n\rightarrow\infty}\sum_{j=0}^n\frac{(-j)_l\alpha^n_j}{(-n)_l}\\
    &\leq& \limsup_{n\rightarrow\infty}\sum_{j=0}^n\frac{(-j)_l\alpha^n_j}{(-n)_l}\leq  \frac{1}{\pi}\int_{c}^{1-c}\frac{x^l}{\sqrt{x}\sqrt{1-x}}+\frac{20\sqrt{c}}{\sqrt{1-c}}.
\end{eqnarray*}
Sending $c\rightarrow 0$ gives
\begin{equation}
    \lim_{n\rightarrow\infty}\sum_{j=0}^n\frac{(-j)_l\alpha^n_l}{(-n)_l}=\frac{1}{\pi}\int_{0}^1\frac{x^l}{\sqrt{x}\sqrt{1-x}}=\frac{(\frac{1}{2})_l}{l!}.
\end{equation}

Now in (\ref{Eq1.1_11}) we take $n\rightarrow\infty$ and get
\begin{equation}
    \lambda_k=\sum_{l=0}^{2k}\frac{(-2k)_l(2k+1)_l(\frac{1}{2})_l}{(l!)^3}=~_3F_2(\frac{1}{2},-2k,2k+1;1,1|1).
\end{equation}
Take $a=-k$ and $b=k+\frac{1}{2}$ in Lemma \ref{L1.1_2}, we obtain that
\begin{equation}
    \lambda_k=({}_2F_1(-k,k+\frac{1}{2};1|1))^2.
\end{equation}
Finally, by Lemma \ref{L1.1_3}, 
\begin{equation}
    {}_2F_1(-k,k+\frac{1}{2};1|1)=\frac{\Gamma(1)\Gamma(\frac{1}{2})}{\Gamma(k+1)\Gamma(\frac{1}{2}-k)}=\frac{(-1)^k 1\cdot 3\cdot\cdots\cdot (2k-1) }{2\cdot 4\cdot\cdots\cdot (2k)}.
\end{equation}
Hence 
\begin{equation}
    \lambda_k=\frac{\binom{2k}{k}^2}{2^{4k}}
\end{equation}
\end{proof}

\subsection{Proof of Theorem \ref{Theorem1}}
In this part, we prove Theorem \ref{Theorem1} based on Theorem \ref{Theorem2}.

\begin{proof}[Proof of Theorem \ref{Theorem1}]
First note that for the starting state $x_0=(1,1,\cdots,1)$, by $S_n$-invariance of $K^l(x,y)$ and $\pi(x)$ for any $x,y,l$, we have
\begin{equation*}
    \|K_{x_0}^l-\pi\|=\frac{1}{2}\sum_{y\in\mathcal{X}} |K_{x_0}^l(y)-\pi(y)|=\frac{1}{2}\sum_{j=0}^n|Q_n^l(j)-u(j)|=\|Q_n^l-u\|,
\end{equation*}
where $u(j)=\frac{1}{n+1},0\leq j\leq n$.

Without loss of generality suppose $n$ is even. The analysis for odd $n$ is similar. 
By Theorem \ref{Theorem2}, 
\begin{equation}
    \lambda_k=\frac{\binom{2k}{k}^2}{2^{4k}}.
\end{equation}
By Lemma \ref{L1.1_1}, we obtain that
\begin{equation}
    \lambda_k\leq \frac{1}{\pi k}(1+\frac{1}{2k})^{\frac{1}{3}}.
\end{equation}

First we show the lower bound. By Proposition \ref{Lower}, taking $\beta=\lambda_1=\frac{1}{4}$ and $\psi=T_2^n$, we obtain
\begin{equation}
    \|Q_n^l-u\| \geq \frac{1}{4}(\frac{1}{4})^l.
\end{equation}

Now we show the upper bound. For $l=1$ the upper bound follows by the fact that total variation distance is always upper bounded by $1$. For $l\geq 2$,
\begin{equation}
    4\|Q_n^l-u\|^2\leq \chi_n^2(l)=\sum_{k=1}^{\frac{n}{2}}\lambda_k^{2l}\beta_{2k}(4k+1),
\end{equation}
where $\beta_{2k}\leq \frac{n}{n+2}\leq 1$ (see \cite[Section 2.5]{DKS} for the estimates related to Hahn polynomials that are used above). Therefore we obtain that 
\begin{eqnarray*}
\chi_n^2(l)&\leq& 5\sum_{k=1}^{\frac{n}{2}}k \lambda_k^{2l}
\leq 15(\frac{1}{16})^l+\sum_{k=3}^{\frac{n}{2}}k(\frac{2}{\pi k})^{2l}\\
&\leq& 15(\frac{1}{16})^l+27(\frac{2}{3\pi})^{2l}\sum_{k=1}^{\infty}\frac{1}{k^2}\leq 60(\frac{1}{16})^l.
\end{eqnarray*}
Hence we conclude that
\begin{equation}
    \|Q_n^l-u\|\leq 4(\frac{1}{4})^l
\end{equation}
for any $l\geq 2$.

Therefore for any $l\in\{1,2,\cdots\}$, 
\begin{equation}
    \frac{1}{4}(\frac{1}{4})^l \leq \|K_{x_0}^l-\pi\|\leq 4(\frac{1}{4})^l. 
\end{equation}
\end{proof}

\section{The twisted Burnside process}\label{Sect.4}
In this section, we consider a generalization of the original Burnside process, which we will call ``twisted Burnside process''. This will allow deformation of the examples above to give a natural process with a larger family of Hahn polynomials as eigenfunctions. Several further examples are presented. We believe that it gives many new examples of easy to run, rapidly converging Markov chains with tunable stationary distributions. 

The setting is the same: we have a finite group $G$ acting on a finite set $\mathcal{X}$; for $x\in\mathcal{X}$, let $G_{x}=\{g\in G:x^g=x\}$; for $g\in G$, let $\mathcal{X}_g=\{x\in \mathcal{X}: x^g=x\}$.

We choose a positive weight $w$ on the group $G$ and let $W(x)$ be the sum of $w(g)$ for $g\in G_x$. We also choose a positive weight $v$ on the set $\mathcal{X}$ and let $V(g)$ be the sum of $v(x)$ for $x\in \mathcal{X}_g$. The new Markov chain is: from $x$, choose $g\in G_x$ with probability $\frac{w(g)}{W(x)}$; given $g$, choose $y\in\mathcal{X}_g$ with probability $\frac{v(y)}{V(g)}$; the chain goes from $x$ to $y$.

Proposition \ref{Pro1} below gives the stationary distribution of the twisted Burnside process.

\begin{proposition}\label{Pro1}
The twisted Burnside process as discussed in the preceding is a reversible Markov chain with stationary distribution 
\begin{equation}
    \pi(x)\propto W(x)v(x)
\end{equation}
for $x\in\mathcal{X}$.

Moreover, if $w$ is constant on each conjugacy class of $G$ and $v$ is constant on each orbit of $\mathcal{X}$ (under the action of $G$), then the chain can be lumped onto orbits of $\mathcal{X}$. For $x\in\mathcal{X}$, let $O_x$ denote the orbit containing $x$. The lumped chain is a reversible Markov chain with stationary distribution
\begin{equation}
    \tilde{\pi}(O_x) \propto \frac{W(x)v(x)}{|G_x|}.
\end{equation}
\end{proposition}
\begin{proof}
For any $x,y\in\mathcal{X}$, the transition probability from $x$ to $y$ of the twisted Burnside process is given by
\begin{equation}\label{E0}
    K_{x,y}=\frac{v(y)}{W(x)}\sum_{g\in G_x\cap G_y}\frac{w(g)}{V(g)}.
\end{equation}
Note that the factor $\sum_{g\in G_x\cap G_y}\frac{w(g)}{V(g)}$ is symmetric in $x,y$. Hence if $\pi(x)=\frac{W(x)v(x)}{Z}$ for $x\in\mathcal{X}$ (where $Z$ is a normalizing constant), then we have
\begin{equation}
    \pi(x)K_{x,y}=\frac{v(x)v(y)}{Z}\sum_{g\in G_x\cap G_y}\frac{w(g)}{V(g)}=\pi(y)K_{y,x}.
\end{equation}
This shows that the twisted Burnside process is reversible with stationary distribution proportional to $W(x)v(x)$.

For the second part, we assume that $w$ is constant on each conjugacy class of $G$ and $v$ is constant on each orbit of $\mathcal{X}$. Below we show by Dynkin's criterion (see \cite[Page 133]{KS}) that in this case the twisted Burnside process can be lumped onto orbits. Suppose that $x,y\in\mathcal{X}$ are in the same orbit. Then there exists $h\in G$ such that $y=x^h$. Hence 
\begin{equation}\label{E1}
    G_y=h^{-1}G_x h.
\end{equation}
As $w$ is constant on each conjugcy class of $G$, by (\ref{E1}) we have
\begin{equation}\label{E2}
    W(y)=\sum_{g\in G_y}w(g)=\sum_{g\in G_x}w(h^{-1}g h)=\sum_{g\in G_x} w(g)=W(x).
\end{equation}
For any $z\in\mathcal{X}$, we have
\begin{eqnarray}\label{E3}
&&\sum_{q\in O_z}v(q)\sum_{g\in G_y\cap G_q}\frac{w(g)}{V(g)}\nonumber\\
&=& \sum_{q\in O_z} v(q^h)\sum_{g\in G_y\cap G_{q^h}}\frac{w(g)}{V(g)}\nonumber\\
&=& \sum_{q\in O_z} v(q)\sum_{g\in h^{-1}(G_x\cap G_q)h}\frac{w(g)}{V(g)}\nonumber\\
&=& \sum_{q\in O_z} v(q)\sum_{g\in G_x\cap G_q}\frac{w(h^{-1} g h)}{V(h^{-1}g h)},
\end{eqnarray}
where the second equality follows from the fact that $v$ is constant on each orbit of $\mathcal{X}$.
Now note that
\begin{equation}\label{E4}
    \mathcal{X}_{h^{-1}g h}=\{x^h: x\in\mathcal{X}_g\}.
\end{equation}
By (\ref{E4}) and the fact that $v$ is constant on each orbit of $\mathcal{X}$, we have
\begin{equation}\label{E5}
    V(h^{-1}g h)=\sum_{x\in \mathcal{X}_{h^{-1}g h}}v(x)=\sum_{x\in\mathcal{X}_g}v(x^h)=\sum_{x\in\mathcal{X}_g}v(x)=V(g).
\end{equation}
By (\ref{E3}), (\ref{E5}) and the fact that $w$ is constant on each conjugacy class of $G$, we have
\begin{eqnarray}\label{E6}
  \sum_{q\in O_z} v(q)\sum_{g\in G_y\cap G_q}\frac{w(g)}{V(g)}
= \sum_{q\in O_z} v(q)\sum_{g\in G_x\cap G_q}\frac{w(g)}{V(g)}.
\end{eqnarray}
Therefore, by (\ref{E0}), (\ref{E2}) and (\ref{E6}), we obtain that
\begin{equation}
 \sum_{q\in O_z}K_{x,q}=\sum_{q\in O_z} K_{y,q}.
\end{equation}
By Dynkin's criterion, the chain can be lumped onto orbits of $\mathcal{X}$. Note that by (\ref{E2}) and the fact that $v(x)$ is constant on each orbit, we have $W(x)v(x)$ is constant on every orbit. Thus the lumped chain is reversible with stationary distribution 
\begin{equation}
    \tilde{\pi}(O_x) \propto |O_x|W(x)v(x).
\end{equation}
By the orbit-stabilizer theorem, we have
\begin{equation}
    \tilde{\pi}(O_x)\propto \frac{W(x)v(x)}{|G_x|}.
\end{equation}
\end{proof}

An example of the twisted version of the Burnside process considered in Theorem \ref{Theorem1} is presented below. Specializing $k=2$ and $\gamma_2=1$ in the example gives the chain with beta-binomial stationary distribution considered in Section \ref{Sect.5} below.

\begin{example}\label{Example1}
Consider $\mathcal{X}=[k]^n$ and $G=S_n$ with $G$ acting on $\mathcal{X}$ by permuting coordinates. Fix $k$ positive parameters $\theta,\gamma_2,\cdots,\gamma_k$. We take 
\begin{equation}
    w(\sigma)=\theta^{c(\sigma)},
\end{equation}
for every $\sigma\in S_n$, where $c(\sigma)$ is the number of cycles of $\sigma$. For any $\vec{x}=(x_1,\cdots,x_n) \in\mathcal{X}$ and $j\in [k]$, we define
\begin{equation}
    S(\vec{x},j):=\#\{i\in [n]: x_i=j\}.
\end{equation}
We further take
\begin{equation}
    v(\vec{x})=\prod_{j=2}^k\gamma_j^{S(\vec{x},j)}
\end{equation}
for every $\vec{x}\in\mathcal{X}$. Below we let $\gamma_1:=1$ to simplify notation.

Now we discuss the twisted Markov chain. From $\vec{x}\in\mathcal{X}$, we choose $\sigma\in S_n$ fixing $\vec{x}$ with probability $\frac{w(\sigma)}{W(\vec{x})}$. This can be realized as follows: find the set of indices $I_j:=\{l\in [n]: x_l=j\}$ for each $j\in [k]$; for each $I_j$, sample a permutation $\sigma_j\in S_{I_j}$ from the Ewens distribution with parameter $\theta$ (see Section \ref{Sect.5} for details of the Ewens distribution); $\sigma$ is the product of $\sigma_j$ for $j\in [k]$. 

Given $\sigma$, we choose $\vec{y}\in\mathcal{X}$ fixed by $\sigma$ with probability $\frac{v(\vec{y})}{V(\sigma)}$. This can be done as below: break $\sigma$ into cycles $C_1,\cdots,C_m$, and denote by $c_d$ the length of the cycle $C_d$ for every $d\in [m]$; for every $d\in [m]$, pick an integer $r_d\in [k]$ with probability $\frac{\gamma_{r_d}^{c_d}}{\sum_{l=1}^k \gamma_l^{c_d}}$, and take $y_i=r_d$ for every $i\in C_d$. Note that for any $\vec{y}\in\mathcal{X}_{\sigma}$, $y_i$ for $i\in C_d$ takes the same value (assuming that it's $r_d$). Hence the probability of generating $\vec{y}$ (where $\vec{y}\in \mathcal{X}_{\sigma}$) through this procedure is  
\begin{equation}
    \prod_{d=1}^m\frac{\gamma_{r_d}^{c_d}}{\sum_{l=1}^k   \gamma_l^{c_d}} \propto  \prod_{d=1}^m\gamma_{r_d}^{c_d}=\prod_{j=2}^k\gamma_j^{S(\vec{y},j)}.
\end{equation}

Note that $w$ is constant on each conjugacy class of $G$. Moreover, $v(\vec{x})$ only depends on $S(\vec{x},j)$ for $j\in [k]$, hence $v$ is constant on each orbit of $\mathcal{X}$. By Proposition \ref{Pro1}, the twisted Markov chain can be lumped onto orbits of $\mathcal{X}$. For $t_1,\cdots,t_k\in\mathbb{N}$ such that $\sum_{l=1}^k t_l=n$, let $\vec{t}:=(t_1,\cdots,t_k)$ denote the orbit of $\mathcal{X}$ consisting of $\vec{x}=(x_1,\cdots,x_n)$ such that $S(\vec{x},j)=t_j$ for every $j\in [k]$. Note that for any $\vec{x}$ in the orbit $\vec{t}$, we have $|G_{\vec{x}}|=\prod_{j=1}^k (t_j)!$ and 
\begin{equation}
    W(\vec{x})=\sum_{\sigma\in S_n: \sigma \text{ fixes }\vec{x}}\theta^{c(\sigma)}=\prod_{j=1}^k(\theta\cdots (\theta+t_j-1)) \propto \prod_{j=1}^k \Gamma(t_j+\theta).
\end{equation}
Thus by Proposition \ref{Pro1}, the stationary distribution of the lumped chain is
\begin{equation}
    \tilde{\pi}(\vec{t})\propto \prod_{j=1}^k\frac{\Gamma(t_j+\theta)}{(t_j)!} \prod_{j=2}^k \gamma_j^{t_j}.
\end{equation}
Note that the term $\prod_{j=1}^k\frac{\Gamma(t_j+\theta)}{(t_j)!}$ is proportional to the probability corresponding to the symmetric Dirichlet-multinomial distribution of parameter $(\theta,\cdots,\theta)$. Thus the twisted Burnside process offers a $k$-dimensional deformation of this classical distribution. Specializing to the case of $k=2$ and choosing parameters so that the base beta-binomial is uniform on $\{0,1,\cdots,n\}$ the deformation is a discrete exponential distribution truncated to this interval.
\end{example}

We close with a final remark. There are many probability measures on $S_n$ that are constant on conjugacy classes. One way to construct these is to define $P(\sigma)$ as proportional to $\theta^{d(\sigma)}$ where $d(\sigma)= d(id,\sigma)$ for $d$ a bi-invariant metric on $S_n$. In turn, such bi-invariant metrics can be constructed as follows: Let $\rho: S_n \rightarrow GL(V)$ be a faithful unitary representation of $S_n$. Let $\|.\|$ be a unitarily invariant norm on $V$. Then $d(\sigma,\tau) = \|\rho(\sigma) - \rho(\tau)\|$ is a bi-invariant metric on $S_n$. In particular, the Cayley distance $d(\sigma, \tau) = n - C(\sigma \tau^{-1}) = \min \# \text{transpositions required to bring } \sigma \text{ to } \tau$ is bi-invariant, giving the example used above. Similarly the Hamming distance $\#\{i \text{ with }\sigma(i)  \text{ different from } \tau(i)\}$ is bi-invariant. A host of other examples appear in \cite[Chapter 6C]{Dia}. This includes von Neuman's useful characterization of unitarily invariant matrix norms.

\section{The twisted Burnside process and Hahn polynomials}\label{Sect.5}
In this section, we add a parameter to the Burnside process on $C_2^n$ with the group $S_n$ so that more general Hahn polynomials appear as eigenfunctions. The idea is simple: replace the uniform distribution on $S_n$, used in step one of the algorithm, by the Ewens distribution
\begin{equation}\label{Ewens}
    \mathbb{P}_{\theta}(\sigma)=\frac{1}{(\theta)_n}\theta^{c(\sigma)}, (\theta)_n=\theta(\theta+1)\cdots (\theta+n-1)
\end{equation}
where $c(\sigma)$ is the number of cycles in $\sigma$ and $0<\theta<\infty$ is a parameter. This familiar distribution is studied in genetics and combinatorics \cite{Crane}. It may be seen as `the Mallows model through the Cayley metric' as discussed at the end of Section \ref{Sect.4}. This chain was discovered via the twisted Burnside construction of Section \ref{Sect.4}. In hindsight, the following simplified description is available. 

On $C_2^n$, from $x\in C_2^n$ with $|x|=\# \text{ ones in }x$
\begin{itemize}
    \item Identify $G_x$ with $S_{|x|}\times S_{n-|x|}$
    \item Pick $\sigma\in S_{|x|}\times S_{n-|x|}$ choosing the two components independently from the Ewens measure (\ref{Ewens})
    \item Break $\sigma$ into cycles and label the cycles $0 \slash 1$ with probability $\frac{1}{2}$. Put this $0\slash 1$ string into $y\in C_2^n$.
\end{itemize}

The argument of Section \ref{Sect.4} shows

\begin{proposition}\label{Station}
The twisted Burnside process given above, lumped to orbits, is a Markov chain on $\{0,1,\cdots,n\}$ with a beta-binomial distribution having parameters $\alpha=\beta=\theta$ (see (\ref{Beta})). 
\end{proposition}

The transition matrix of this Markov chain, call it $p_{ij}^{n,\theta}$, can be written explicitly.

\begin{proposition}\label{Pro}
Consider the twisted Burnside process given above, lumped to orbits. The transition matrix is given by
\begin{equation}
   p^{n,\theta}_{0j}=\binom{n}{j}\frac{\frac{\theta}{2}\cdots(\frac{\theta}{2}+j-1)\frac{\theta}{2}\cdots(\frac{\theta}{2}+n-j-1)}{\theta(\theta+1)\cdots(\theta+n-1)},
\end{equation}
\begin{equation}
    p^{n,\theta}_{jk}=\sum_{\max\{0,j+k-n\}\leq l\leq\min\{j,k\}}p^{j,\theta}_{0l}p^{n-j,\theta}_{0,k-l}.
\end{equation}
\end{proposition}
\begin{proof}
We prove this using P\'olya's cycle index theorem (see (\ref{Pol1}),(\ref{Pol2})). Note that we have
\begin{equation}
    p^{n,\theta}_{0j}=\frac{n!}{\theta(\theta+1)\cdots(\theta+n-1)}\frac{1}{n!}\sum_{g\in S_n,\lambda\vdash j}\prod_{i=1}^j\binom{a_i(g)}{b_i(\lambda)}(\frac{\theta}{2})^{a_1(g)+\cdots+a_n(g)}.
\end{equation}
Using P\'olya's cycle index theorem by taking derivatives and multiplying, we get
\begin{equation}
    p^{n,\theta}_{0j}=\binom{n}{j}\frac{\frac{\theta}{2}\cdots(\frac{\theta}{2}+j-1)\frac{\theta}{2}\cdots(\frac{\theta}{2}+n-j-1)}{\theta(\theta+1)\cdots(\theta+n-1)}.
\end{equation}
Moreover, by the definition of the twisted Burnside process given above, we have
\begin{equation}
    p^{n,\theta}_{jk}=\sum_{\max\{0,j+k-n\}\leq l\leq\min\{j,k\}}p^{j,\theta}_{0l}p^{n-j,\theta}_{0,k-l}.
\end{equation}
\end{proof}

Theorems \ref{Theorem1} and \ref{Theorem2} above 
deform in the following form.

\begin{theorem}\label{Theorem3}
Consider the twisted Burnside process on $C_2^n$ given above, lumped to orbits. The non-zero eigenvalues of the Markov chain are given by $1$ and 
\begin{equation}
    \lambda_k=~_3 F_2(-2k,2k+2\theta-1,\frac{\theta}{2};\theta,\theta|1)
\end{equation}
for $1\leq k\leq \frac{n}{2}$.
Moreover, the eigenfunctions corresponding to the zero eigenvalues are the Hahn polynomials on $\{0,1,\cdots,n\}$ with parameters $\alpha=\beta=\theta$ of odd degree. The eigenfunction corresponding to the eigenvalue $\lambda_k$ is the Hahn polynomial on $\{0,1,\cdots,n\}$ with parameters $\alpha=\beta=\theta$ of degree $2k$, $1\leq k\leq\frac{n}{2}$.
\end{theorem}

\begin{theorem}\label{Theorem4}
Consider the twisted Burnside process on $C_2^n$ given above with $\theta\geq 1$. Denote by $\pi$ the stationary distribution of the chain (see Proposition \ref{Station}), and denote by $K_{x_0}^l$ for $x_0=(1,1,\cdots,1)$ the distribution after $l$ steps starting from $n$ ones. Then there exist positive constants $c(\theta),C(\theta)$ which only depend on $\theta$, such that for all $n\geq 2$ and all $l\geq 1$,
\begin{equation}
    c(\theta)(\frac{1}{2(1+\theta)})^l\leq\|K_{x_0}^l-\pi\|\leq C(\theta)(\frac{1}{2(1+\theta)})^l.
\end{equation}
\end{theorem}

The proofs of Theorems \ref{Theorem3} and \ref{Theorem4} are similar but quite a bit more involved, to the proofs of Theorems \ref{Theorem1} and \ref{Theorem2}. The restriction that $\theta\geq 1$ in Theorem \ref{Theorem4} is due to a technical step in our proof (for certain estimates of the eigenvalues). We refer the interested reader to \cite{Zho1,Zho2}. This develops things for $k\geq 2$ and has other approaches to proof.

We have not (yet) succeeded in finding a two-parameter deformation of the Burnside process on $C_2^n$ which gives the full set of Hahn polynomials as eigenfunctions. Similarly, we have not succeeded in diagonalizing the Burnside process on $[k]^n$ for any $k\geq 3$.

The point of this paper was to show (a) that orthogonal polynomials `pop up' everyplace (b) seeing orthogonal polynomials as belonging to families leads to useful extension of classical algorithms.

We are sorry not to be able to ask Dick Askey for further help.

\section{A continuous limit of the Burnside process}

A referee has made the welcome suggestion that we try to `pass to the limit' going from our discrete version of the Burnside process analyzed above to a continuous process. While we don't know a continuous version of the Cauchy-Frobenius Lemma or P\'olya theory, we were able to pass to the limit and this proved informative.

As motivation, recall that the discrete process begins with a point $x\in C_2^n$. A permutation in $S_k\times S_{n-k}$ is chosen at random (there are $k$ ones in $x$), split into cycles and these are labeled $0\slash 1$ to give $y\in C_2^n$. As explained, only the number of ones enters, not their positions. So the process can be thought of as taking place on $\{0,1,2,\cdots,n\}$. We divide by $n$ and form a process on $[0,1]$. The analog of the cycles of a random permutation is replaced by a stick-breaking process on $[0,1]$ familiar from the Chinese restaurant process and Dirichlet random measures (\cite{ABT,Kin,Set}). Combining gives the following Markov chain on $[0,1]$:

From $x\in [0,1]$, break the interval $[0,x]$ into countably many pieces by a stick-breaking process. Namely, let $R_1,R_2,\cdots$ be independent $Beta(1,\theta)$ random variables, and define $Y_1=xR_1$ and $Y_j=x(1-R_1)\cdots(1-R_{j-1})R_j$ for every $j=2,3,\cdots$; then we break $[0,x]$ into pieces of lengths $Y_1,Y_2,\cdots$. Break the interval from $x$ to $1$ in the same way. Label each interval $0\slash 1$ by flipping a fair coin. Let $y$ be the total length of the pieces labeled $1$. This gives a Markov chain on $[0,1]$. It is a natural limiting version of our discrete Burnside process.

This Markov chain can be equivalently described as follows (\cite[Section 3]{DK},\cite{Let,Set}). From $x\in [0,1]$, sample two independent $Beta(\frac{\theta}{2},\frac{\theta}{2})$ random variables $Z,Z'$. Let $y=xZ+(1-x)Z'$, and move to $y$.

\begin{theorem}\label{Theorem5}
The Markov chain above is reversible with $Beta(\theta,\theta)$ stationary distribution. The non-zero eigenvalues are given by $1$ and
\begin{equation}\label{Exp1}
    \lambda_k=~_3F_2(-2k,2k+2\theta-1,\frac{\theta}{2};\theta,\theta|1)
\end{equation}
for $k=1,2,\cdots$. An alternative expression for $\lambda_k$ is given by
\begin{equation}\label{Exp2}
    \lambda_k=\mathbb{E}[(Z-Z')^{2k}],
\end{equation}
where $Z,Z'$ are two independent $Beta(\frac{\theta}{2},\frac{\theta}{2})$ random variables.

Moreover, the eigenfunctions corresponding to the zero eigenvalues are the Jacobi polynomials associated to the stationary distribution of odd degree. For every $k=1,2,\cdots$, the eigenfunction corresponding to the eigenvalue $\lambda_k$ is the Jacobi polynomial associated 
to the stationary distribution of degree $2k$.
\end{theorem}

\begin{proof}
From the definition of the Markov chain, the transition density is given by
\begin{eqnarray}
    k(x,y)&=& \int_{\max\{0,x+y-1\}}^{\min\{x,y\}}\frac{z^{\frac{\theta}{2}-1}(x-z)^{\frac{\theta}{2}-1}}{B(\frac{\theta}{2},\frac{\theta}{2})} \frac{(y-z)^{\frac{\theta}{2}-1}(1-x-y+z)^{\frac{\theta}{2}-1}}{B(\frac{\theta}{2},\frac{\theta}{2})}dz\nonumber\\
    &&\times x^{1-\theta}(1-x)^{1-\theta}
\end{eqnarray}
for $x,y\in [0,1]$. Let $\pi(x)=\frac{x^{\theta-1}(1-x)^{\theta-1}}{B(\theta,\theta)},x\in [0,1]$ be the probability density function of $Beta(\theta,\theta)$.
We have
\begin{equation*}
    \pi(x)k(x,y)=\pi(y)k(y,x)
\end{equation*}
for any $x,y\in (0,1)$. Therefore the Markov chain is reversible with $Beta(\theta,\theta)$ stationary distribution.

Suppose that we start from $x\in [0,1]$. Sample two independent $Beta(\frac{\theta}{2},\frac{\theta}{2})$ random variables $Z,Z'$, and let $y=xZ+(1-x)Z'$. This gives one iteration of the Markov chain. For every $l=1,2,\cdots$
\begin{equation}\label{En1}
    \mathbb{E}[y^l|x]=\mathbb{E}[(Z-Z')^l]x^l+\sum_{j=0}^{l-1}\binom{l}{j}\mathbb{E}[(Z-Z')^{j}(Z')^{l-j}]x^j.
\end{equation}
The right hand side of (\ref{En1}) is a polynomial in $x$ of degree $\leq l$ with leading coefficient given by $\mathbb{E}[(Z-Z')^l]$. When $l$ is odd, as the distribution of $Z-Z'$ is symmetric around $0$, we have $\mathbb{E}[(Z-Z')^l]=0$. 

By Cannings argument (the analogue of Lemma \ref{L1.1_4}), the non-zero eigenvalues of the Markov chain are given by $1$ and $\lambda_k:=\mathbb{E}[(Z-Z')^{2k}]$ for $k=1,2,\cdots$. Moreover, the eigenfunctions corresponding to the zero eigenvalues are the Jacobi polynomials associated to $Beta(\theta,\theta)$ of odd degree, and the eigenfunction corresponding to the eigenvalue $\lambda_k$ is the Jacobi polynomial associated 
to $Beta(\theta,\theta)$ of deree $2k$ for every $k=1,2,\cdots$.

Finally, we show that 
\begin{equation*}
    \lambda_k=~_3F_2(-2k,2k+2\theta-1,\frac{\theta}{2};\theta,\theta|1).
\end{equation*}
We denote by
\begin{equation}
    \phi_{2k}(x)=~_2 F_1(-2k,2\theta+2k-1;\theta|x)=\sum_{l=0}^{2k}\frac{(-2k)_l(2\theta+2k-1)_l }{(\theta)_l}\frac{x^l}{l!}
\end{equation}
the Jacobi polynomial associated to $Beta(\theta,\theta)$ of degree $2k$, normalized so that $\phi_{2k}(0)=1$. Note that the one-step distribution starting from $0$ follows the $Beta(\frac{\theta}{2},\frac{\theta}{2})$ distribution. Thus letting $T\sim Beta(\frac{\theta}{2},\frac{\theta}{2})$, we have
\begin{equation}
    \mathbb{E}[\phi_{2k}(T)]=\lambda_k\phi_{2k}(0)=\lambda_k.
\end{equation}
For every $l=0,1,2,\cdots$
\begin{equation}
    \mathbb{E}[T^l]=\frac{1}{B(\frac{\theta}{2},\frac{\theta}{2})}\int_{0}^1 x^{l+\frac{\theta}{2}-1}(1-x)^{\frac{\theta}{2}-1}dx=\frac{(\frac{\theta}{2})_l}{(\theta)_l}.
\end{equation}
Hence
\begin{eqnarray*}
\lambda_k&=&\mathbb{E}[\phi_{2k}(T)] = \sum_{l=0}^{2k}\frac{(-2k)_l(2\theta+2k-1)_l}{(\theta)_l}\frac{\mathbb{E}[T^l]}{l!}\\
&=&~_3 F_2(-2k,2k+2\theta-1,\frac{\theta}{2};\theta,\theta|1). 
\end{eqnarray*}

\end{proof}

\begin{remark}
Comparison with Theorems \ref{Theorem2} and \ref{Theorem3} shows this limit captures the essential features we encountered. Note that the eigenvalue $\lambda_k$ here matches that of the discrete chain in Theorem \ref{Theorem3} as long as $n\geq 2k$.
\end{remark}

\begin{remark}
As $Z,Z'\in [0,1]$, we have $|Z-Z'|\leq 1$. Hence for any $k=1,2,\cdots$
\begin{equation*}
    \lambda_{k}=\mathbb{E}[|Z-Z'|^{2k}]\geq \mathbb{E}[|Z-Z'|^{2k+2}]=\lambda_{k+1}.
\end{equation*}
Therefore the eigenvalues $\lambda_k$ are monotone decreasing.
\end{remark}

\begin{remark}
The two expressions (\ref{Exp1}) and (\ref{Exp2}) for $\lambda_k$ lead to the following identity
\begin{equation}
    \sum_{l=0}^{2k}\frac{(-2k)_l(\frac{\theta}{2})_l(\frac{\theta}{2})_{2k-l}}{(\theta)_l(\theta)_{2k-l}l!}= \sum_{l=0}^{2k}\frac{(-2k)_l(\frac{\theta}{2})_l(2k+2\theta-1)_l}{(\theta)_l(\theta)_{l}l!}.
\end{equation}
We didn't know this identity but Dennis Stanton observes that it is a special case of the following transformation
\begin{equation}\label{Tr}
    (a+A)_n~_3F_2(a,c-b,-n;c,a+A|1)=(A)_n~_3F_2(a,b,-n;c,1-A-n|1)
\end{equation}
when $n=2k,a=\frac{\theta}{2},c=\theta,b=2\theta+2k-1,A=1-\theta-2k$. One of his proofs of (\ref{Tr}) proceeds by multiplying the Pfaff transformation (\cite[Page 43]{Koepf})
\begin{equation*}
    (1-x)^{-a}~_2F_1(a,c-b;c|\frac{x}{x-1})=~_2F_1(a,b;c|x)
\end{equation*}
by $(1-x)^{-A}$ and equating coefficients of $x^n$.

\end{remark}

\newpage
\bibliographystyle{spmpsci}
\bibliography{Burnside.bib}

\end{document}